\documentclass[a4paper, english, 12pt, reqno]{amsart}

\usepackage{amssymb}
\usepackage[english]{babel}
\usepackage[utf8]{inputenc}
\usepackage{tikz,tikzscale}
\tikzset{>=latex}
\usepackage{xcolor}
\usepackage{enumerate}
\usepackage{booktabs,multirow}

\newtheorem{theorem}{Theorem}[section]
\newtheorem{lemma}[theorem]{Lemma}

\theoremstyle{remark}

\numberwithin{equation}{section}

\newcommand{\elem}{\ensuremath{T}}
\newcommand{\mesh}{\ensuremath{\mathcal T}}

\newcommand{\faceSet}{\ensuremath{\mathcal F}}
\newcommand{\faceSetDir}{\ensuremath{\mathcal F^\textup D}}
\newcommand{\faceNum}{\#\faceSet^\elem}
\newcommand{\face}{\ensuremath{F}}
\newcommand{\skeletalSpace}{\ensuremath{M}}
\newcommand{\contElementSpace}{\ensuremath{V^\textup c}}
\newcommand{\linElementSpace}{\ensuremath{\overline V^\textup c}}
\newcommand{\discElementSpace}{\ensuremath{V}}
\newcommand{\polynomials}{\ensuremath{\mathcal P}}
\newcommand{\level}{\ensuremath{\ell}}
\newcommand{\iterMgOuter}{i}%\ensuremath{\nu}}
\newcommand{\iterMgInner}{m}%\ensuremath{\eta}}

\newcommand{\Div}{\nabla\!\cdot\!}

\newcommand{\extensionOp}{\ensuremath{\mathcal U^\textup c}}
\newcommand{\averagingOp}{\ensuremath{I^\textup{avg}}}
\newcommand{\traceOp}{\ensuremath{\gamma_\level}}
\newcommand{\injectionOp}{\ensuremath{I}}
\newcommand{\projectionOp}{\ensuremath{P}}
\newcommand{\projectionOrthogonalOP}{\ensuremath{\Pi}}
\newcommand{\projectionLinOp}{\ensuremath{\overline P}}

\newcommand{\tildelambda}{\tilde\lambda}
\newcommand{\ureconstructed}{\overline{u}}

\renewcommand{\vec}[1]{\ensuremath{\boldsymbol{#1}}}
\newcommand{\Nu}{\ensuremath{\vec \nu}}
\newcommand{\dx}{\ensuremath{\, \textup d x}}
\newcommand{\ds}{\ensuremath{\, \textup d \sigma}}
\newcommand{\localU}{\ensuremath{\mathcal U}}
\newcommand{\localQ}{\ensuremath{\vec{\mathcal Q}}}

\newcommand{\llangle}{\ensuremath{\langle \! \langle}}
\newcommand{\rrangle}{\ensuremath{\rangle \! \rangle}}
\newcommand{\nnorm}{\ensuremath{\vert \! \vert \! \vert}}

\usepackage[colorlinks = true, linkcolor = blue, citecolor = blue, urlcolor = blue]{hyperref}

\begin{document}

\title{HMG --- Homogeneous multigrid for HDG} 

\author{Peipei Lu}
\address{Department of Mathematics Sciences, Soochow University, Suzhou, 215006, China}
% \curraddr{}
\email{pplu@suda.edu.cn}
\thanks{P.~Lu has been supported by the Alexander von Humboldt Foundation.}

\author{Andreas Rupp}
\address{Interdisciplinary Center for Scientific Computing (IWR), Heidelberg University, Mathematikon, Im Neuenheimer Feld 205, 69120 Heidelberg, Germany}
% \curraddr{}
\email{andreas.rupp@fau.de, andreas.rupp@uni-heidelberg.de}
\thanks{This work is supported by the Deutsche Forschungsgemeinschaft (DFG, German Research Foundation) under Germany's Excellence Strategy EXC 2181/1 - 390900948 (the Heidelberg STRUCTURES Excellence Cluster).}

\author{Guido Kanschat}
\address{Interdisciplinary Center for Scientific Computing (IWR) and Mathematics Center Heidelberg (MATCH), Heidelberg University, Mathematikon, Im Neuenheimer Feld 205, 69120 Heidelberg, Germany}
% \curraddr{}
\email{kanschat@uni-heidelberg.de}
% \thanks{}


\date{\today}

% \dedicatory{}

\begin{abstract}
 We introduce a homogeneous multigrid method in the sense that it uses the same HDG discretization scheme for Poisson's equation on all levels. In particular, we construct a stable injection operator and prove optimal convergence of the method under the assumption of elliptic regularity. Numerical experiments underline our analytical findings.
\end{abstract}

\maketitle

% ---------------------------------------------------------------------------

% Content of the article.

% \begin{envgknote}{Using FixMe}
%   You can use commands \verb!\plnote{..}! and \verb!\arnote{...}! for
%   annotations, where ``pl'' stands for Peipei Lu and ``ar'' for
%   Andreas Rupp. Different colors can be obtained replacing ``note'' by
%   ``warning'', ``error'', or ``fatal''. The corresponding environments
%   are
%   \begin{center}
%     \verb!\begin{envplnote}{short}...\end{envplnote}!
% 
%     \verb!\begin{envarnote}{short}...\end{envarnote}!
%   \end{center}
%   They typeset the note in the text, not on the margin.
% 
%   Suggestions for the use of different types:
%   \begin{description}
%   \item[note] Comments and questions
%   \item[warning] Suggest am addition, also reminder to self
%   \item[error] Suggest change
%   \item[fatal] Is this wrong?
%   \end{description}
% \end{envgknote}

\section{Introduction}
While hybridizable discontinuous Galerkin (HDG) methods have become
popular in mathematics and applications over the last years, literature
on efficient solution of the resulting discrete systems is still
scarce. In this article, we propose to our knowledge the first
multigrid preconditioner for such methods which is homogeneous in the
sense that it uses the same discretization scheme on all levels. Such
methods are important, since they have the same mathematical
properties on all levels. They can also be advantageous from a
computational point of view, since their data structures and execution patterns are more regular.

HDG methods have been gaining popularity in the last
decade. Originally, they were analyzed for the Laplacian, see for
instance~\cite{CockburnGL2009} for an overview. Meanwhile, they have
been applied to stationary~\cite{CockburnG2009,CockburnNP2020,CockburnSayas14,EggerWaluga13,NguyenPC2010,oikawa2016analysis} and
instationary~\cite{LehrenfeldSchoeberl16} Stokes systems, to the
locking-free discretization of problems in
elasticity~\cite{FuLehrenfeldLinkeStreckenbach20arxiv}, as well as to
plates~\cite{CockburnDongGuzman09,HuangHuang19} and beams~\cite{CelikerCockburnShi10}.
However, only few results have been achieved for solving the large systems
of equations arising from this type of numerical method, while
multigrid \cite{GopaK2003} and domain
decomposition \cite{FengKarakashian01} solvers for earlier
discontinuous Galerkin (DG) schemes have been available for many
years.

The first multigrid method for HDG discretizations was introduced in~\cite{CockburnDGT2013,TanPhD}.
It is basically a two-level method, where the ``coarse space'' consists of the piecewise
linear, conforming finite element space. The coarse grid solver consists of a 
conforming multigrid method for lowest order elements. In \cite{GopalakrishnanTan09}, similar
results have been applied to hybridized mixed (e.g. Raviart--Thomas (RT)) elements.

A BPX preconditioner for non-standard finite element methods, including hybridized RT, BDM, the weak Galerkin, and Crouzeix--Raviart methods, is analyzed in \cite{li2016bpx} and a two level algorithm for HDG methods for the diffusion problem is presented in \cite{li2016analysis}. All of the aforementioned  methods utilize the piecewise linear conforming finite element space as the auxiliary space.

Two-level analysis of HDG methods seems equally rare. A method for high order HDG discretizations is introduced in~\cite{SchoeberlLehrenfeld13}. Here, the authors focus on showing that standard $p$-version domain decomposition techniques can be applied and prove condition number estimates on tetrahedral meshes polylogarithmic in $p$.
More recently, a domain decomposition for a hybridized Stokes problem
was presented in~\cite{BarrenecheaBosyDoleanNatafTournier18}, but it
does not discuss a coarse space. Moreover, \cite{HeRS2020} discusses local Fourier analysis of interior penalty based multigrid methods for tensor-product type polynomials in two spatial dimensions.

The methods known so far are \emph{heterogeneous} in the sense that the multigrid cycle is not performed on the HDG discretization itself, but on a surrogate scheme.
In view of future generalization beyond second order
elliptic problems, we decided to devise a \emph{homogeneous} method which uses the same
discretization scheme on all levels, and thus only employs a single finite
element method. The analysis uses the abstract framework
developed for noninherited forms in~\cite{DuanGTZ2007}. It is based on
arguments found in~\cite{Gopa2003,TanPhD}. Nevertheless, the coarse grid
operator is genuinely HDG and of the same type as the fine grid
operator. Only the injection operator from coarse to fine level uses
continuous interpolation in an intermediate step.

Since we focus on new coarse spaces and intergrid operators, we rely on standard smoothers and analyze a standard Poisson problem with elliptic regularity to present the basic ideas of our method. However, the regularity assumtion could be weakened utilizing the ideas of~\cite{BramblePX1991}, of~\cite{Brenner1999} (which analyzes nonconforming multigrid methods), of~\cite{GopaK2003} (which deals with multigrid methods for DG), and of~\cite{GopalakrishnanTan09} at the cost of more technicalities within our proofs. As for robust smoothers for higher order methods, this will be subject to further research.

The remainder of this paper is structured as follows: In Section \ref{SEC:basics}, we review the HDG method for elliptic PDEs. Furthermore, an overview over the  used function spaces, scalar products, and operators is given. Section \ref{SEC:multigrid} is devoted to the definition of the multigrid and states its main convergence result. Sections \ref{SEC:proof_a1} and \ref{SEC:proof_a2_a3} verify the assumptions of the main convergence result, while Section \ref{SEC:numerics} underlines its validity by numerical experiments. Short conclusions wrap up the paper.

% \begin{itemize}
%  \item \cite{oikawa2015hybridized}: \textcolor{red}{Remove this?!}
% \end{itemize}
% 
% ---------------------------------------------------------------------
\section{Model equation and discretization}\label{SEC:basics}
% ---------------------------------------------------------------------
% 
We consider the standard diffusion equation in mixed form defined on a polygonally bounded Lipschitz domain $\Omega \subset \mathbb R^d$ with boundary $\partial \Omega$. We assume homogeneous Dirichlet boundary conditions on $\partial\Omega$. Thus, we approximate solutions $(u, \vec q)$ of
\begin{subequations}\label{EQ:diffusion_mixed}
\begin{align}
 \Div \vec q & = f && \text{ in } \Omega,\\
 \vec q + \nabla u & = 0 && \text{ in } \Omega,\\
 u & = 0 && \text{ on } \partial \Omega,%\\
%  \vec q \cdot \Nu & = g_\text N && \text{ on } \Gamma_\text N,
\end{align}
\end{subequations}
for a given function $f$. In the analysis, we will assume elliptic regularity, namely $u \in H^2(\Omega)$ if $f \in L^2(\Omega)$, such there is a constant $c>0$ for which holds
\begin{gather}
  |u|_{H^2(\Omega)}
  \le c \| f \|_{L^2(\Omega)}.
\end{gather}
Here, and in the following, $L^2(\Omega)$ denotes the space of square integrable functions on $\Omega$ with inner product and norm
\begin{equation}
 (u,v)_0 := \int_\Omega u v \dx, \qquad \text{and} \qquad \| u \|^2_0 := (u,u)_0.
\end{equation}
The space $H^k(\Omega)$ is the Sobolev space of $k$-times weakly
differentiable functions with derivatives in $L^2(\Omega)$. We note
that the assumption of homogeneous boundary data was introduced for
simplicity of presentation and can be lifted by standard arguments.

% ---------------------------------------------------------------------
\subsection{Spaces for the HDG multigrid method}
% ---------------------------------------------------------------------
%
Starting out from a subdivision $\mesh_0$ of $\Omega$ into simplices,
we construct a hierarchy of meshes $\mesh_\level$ for
$\level = 1,\dots, L$ recursively by refinement, such that each cell
of $\mesh_{\level-1}$ is the union of several cells of mesh
$\mesh_\level$. We assume that the mesh is regular, such that each facet
of a cell is either a facet of another cell or on the
boundary. Furthermore, we assume that the hierarchy is shape regular
and thus the cells are neither anisotropic nor otherwise distorted.
We call $\level$ the level of the quasi-uniform mesh $\mesh_\level$ and denote by
$h_\level$ the characteristic length of its cells. We assume that
refinement from one level to the next is not too fast, such that there
is a constant $c_\text{ref} > 0$ with
\begin{gather}
  h_\level \ge c_\text{ref} h_{\level-1}.
\end{gather}

By $\faceSet_\level$ we denote the set of faces of $\mesh_\level$.
The subset of faces on the boundary is
\begin{gather}
  \faceSetDir_\level := \{\face \in \faceSet_\level : \face \subset \partial \Omega \}.
\end{gather}
Moreover, we define
$\faceSet^\elem_\level := \{ \face \in \faceSet_\level : \face \subset
\partial \elem \}$ as the set of faces of a cell $\elem\in\mesh_\level$.  On the set of faces, we define the space $L^2(\faceSet_\ell)$ as the space of
square integrable functions with the inner product
\begin{gather}
  \llangle \lambda, \mu \rrangle_\level = \sum_{\elem \in \mesh_\level} \int_{\partial \elem} \lambda\mu\ds,
\end{gather}
and its induced norm
$\nnorm \mu \nnorm^2_\level = \llangle \mu, \mu
\rrangle_\level$. Note that interior faces appear twice in this definition such that expressions like $\llangle u, \mu \rrangle_\level$ with possibly discontinuous $u|_{\elem} \in H^1(\elem)$ for all $\elem \in \mesh_\level$ and $\mu \in L^2(\faceSet_\level)$ are defined without further ado. Additionally, we define an inner product
commensurate with the $L^2$-inner product in the bulk domain, namely
\begin{gather}
  \langle \lambda, \mu \rangle_\level
  = \sum_{\elem \in \mesh_\level} \frac{|\elem|}{|\partial \elem|}
  \int_{\partial \elem} \lambda \mu \ds \cong \sum_{\face \in \faceSet_\level} h_\face
  \int_{\face} \lambda \mu \ds.
\end{gather}
Its induced norm is $ \| \mu \|^2_\level = \langle \mu, \mu \rangle_\level$.

Let $p\ge 1$ and $\polynomials_p$ be the space of (multivariate)
polynomials of degree up to $p$. Then, we define the space of piecewise
polynomials on the skeleton by
\begin{gather}
  \label{EQ:skeletal_space}
  \skeletalSpace_\level := \left\{ \lambda \in L^2(\faceSet_\level) \;\middle|\;
    \begin{array}{r@{\,}c@{\,}ll}
  \lambda_{|\face} &\in& \polynomials_p & \forall \face \in \faceSet_\level\\
  \lambda_{|\face} &=& 0 & \forall \face \in \faceSetDir_\level    
    \end{array}
  \right\}.
\end{gather}

The HDG method involves a local solver on each mesh cell
$\elem \in \mesh_\level$, producing cellwise approximations $u_\elem \in V_\elem$
and $\vec q_\elem\in \vec W_\elem$ of the functions $u$ and $\vec q$ in
equation~\eqref{EQ:diffusion_mixed}, respectively. We choose
$V_\elem = \polynomials_p$. Then, choosing
$\vec W_\elem = \polynomials_p^d$ yields the so called hybridizable local discontinuous Galerkin (LDG-H) scheme. Our current analysis is in fact limited to this case and other choices require a modification of Lemma~\ref{LEM:u_q_bound}. We will also use the concatenations of the spaces $V_\elem$
and $\vec W_\elem$, respectively, as a function space on $\Omega$, namely
\begin{gather}
  \label{EQ:dg_spaces}
  \begin{aligned}
    \discElementSpace_\level
    &:=\bigl\{ v \in L^2(\Omega)
    & \big|\;v_{|\elem} &\in V_\elem,
    &\forall \elem &\in \mesh_\level \bigr\},\\
    \vec W_\level
    &:=\bigl\{ \vec q \in L^2(\Omega;\mathbb R^d)
    & \big|\;\vec q_{|\elem} &\in \vec W_\elem,
    &\forall \elem &\in \mesh_\level \bigr\}.    
  \end{aligned}
\end{gather}
%\begin{alignat}2
% \end{alignat}
% 
% ---------------------------------------------------------------------
\subsection{Hybridizable discontinuous Galerkin method for the diffusion equation}\label{SEC:HDG_definition}
% ---------------------------------------------------------------------
% 
The HDG scheme for~\eqref{EQ:diffusion_mixed} on a mesh $\mesh_\level$
consists of a local solver and a global coupling equation. The local
solver is defined cellwise by a weak formulation
of~\eqref{EQ:diffusion_mixed} in the discrete spaces
$V_\elem \times \vec W_\elem$ and defining suitable numerical traces and fluxes. Namely, given
$\lambda \in \skeletalSpace_\level$ find $u_\elem \in V_\elem$ and
$\vec q_\elem \in \vec W_\elem$ , such that
\begin{subequations}\label{EQ:hdg_scheme}
\begin{align}
  \int_\elem \vec q_\elem \cdot \vec p_\elem \dx - \int_\elem u_\elem \Div \vec p_\elem \dx
  & = - \int_{\partial \elem} \lambda \vec p_\elem \cdot \Nu \ds
    \label{EQ:hdg_primary}
  \\
  - \int_\elem \vec q_\elem \cdot \nabla v_\elem \dx  + \int_{\partial \elem} ( \vec q_\elem \cdot \Nu + \tau_\level u_\elem ) v_\elem \ds
  & = \tau_\level \int_{\partial \elem} \lambda v_\elem \ds \label{EQ:hdg_flux}
\end{align}
\end{subequations}
hold for all $v_\elem \in V_\elem$, and all $\vec p_\elem \in \vec W_\elem$, and for
all $\elem \in \mesh_\level$. Here,
$\Nu$ is the outward unit normal with respect to $\elem$ and $\tau_\level > 0$
is the penalty coefficient. While the local solvers are implemented
cell by cell, it is helpful for the analysis to combine them by
concatenation. Thus, the local solvers define a mapping
\begin{gather}
  \begin{split}
    \skeletalSpace_\level & \to \discElementSpace_\level \times \vec W_\level\\
   \lambda &\mapsto (\localU_\level \lambda, \localQ_\level \lambda),
 \end{split}
\end{gather}
where for each cell $\elem\in \mesh_\level$ holds
$\localU_\level \lambda = u_\elem$ and
$ \localQ_\level \lambda = \vec q_\elem$. In the same way, we define
operators $\localU_\level f$ and $ \localQ_\level f$ for
$f\in L^2(\Omega)$, where now the local solutions are defined by
the system
\begin{subequations}\label{EQ:hdg_f}
  \begin{align}
    \int_\elem \vec q_\elem \cdot \vec p_\elem \dx - \int_\elem u_\elem \Div \vec p_\elem \dx
    & = 0
      \label{EQ:hdg_f_primary}
    \\
    - \int_\elem \vec q_\elem \cdot \nabla v_\elem \dx  + \int_{\partial \elem} ( \vec q_\elem \cdot \Nu + \tau_\level u_\elem ) v_\elem \ds
    & =  \int_{\elem} f v_\elem \dx.
      \label{EQ:hdg_f_flux}
\end{align}
\end{subequations}

Once $\lambda$ has been computed, the HDG approximation
to~\eqref{EQ:diffusion_mixed} on mesh $\mesh_\level$ will be computed as
\begin{gather}
  \begin{split}
    u_\level &= \localU_\level \lambda + \localU_\level f\\
    \vec q_\level &= \localQ_\level \lambda + \localQ_\level f
  \end{split}
\end{gather}

The global coupling condition is derived through a discontinuous
Galerkin version of mass balance and reads: Find
$\lambda \in \skeletalSpace_\level$, such that for all
$ \mu \in \skeletalSpace_\level$
\begin{equation}
  \sum_{\elem \in \mesh_\level}
  \sum_{\face \in \faceSet^\elem_\level \setminus \faceSetDir_\level}
   \int_\face \left( \vec q_\level \cdot \Nu
    + \tau_\level (u_\level - \lambda)\right) \mu \ds = 0.\label{EQ:hdg_global}
\end{equation}

In~\cite{CockburnGL2009}, it is shown that $\lambda \in \skeletalSpace_\level$ is the solution of the coupled system from~\eqref{EQ:hdg_scheme} to~\eqref{EQ:hdg_global} if and only if it is the solution of
\begin{subequations}\label{EQ:hdg_condensed}
\begin{equation}\label{EQ:hdg_condensed_forms}
 a_\level (\lambda, \mu) = b_\level(\mu) \qquad \forall \mu \in \skeletalSpace_\level,
\end{equation}
with
\begin{align}
 a_\level(\lambda, \mu) = & \int_\Omega \localQ_\level \lambda \localQ_\level \mu \dx + \sum_{\elem \in \mesh_\level} \int_{\partial \elem} \tau_\level (\localU_\level \lambda - \lambda) (\localU_\level \mu - \mu) \ds \label{EQ:bilinear_condensed},\\
 b_\level(\mu) = & \int_\Omega \localU_\level \mu f \dx.
\end{align}
\end{subequations}
Furthermore, the bilinear form $a_\level(\lambda, \mu)$ is symmetric and positive definite.
Thus, it induces a norm
\begin{gather}
  \| \mu \|^2_{a_\level} = a_\level(\mu, \mu),
\end{gather}

We close this subsection by associating an operator
$A_\ell\colon \skeletalSpace_\level \to \skeletalSpace_\level$ with
the bilinear form $a_\level(\cdot,\cdot)$ by the relation
\begin{gather}
 \langle A_\level \lambda, \mu \rangle_\level = a_\level(\lambda, \mu) \qquad \forall \mu \in \skeletalSpace_\level.
\end{gather}

%%%%%%%%%%%%%%%%%%%%%%%%%%%%%%%%%%%%%%%%%%%%%%%%%%%%%%%%%%%%%%%%%%%%%%%%%%%%%%%%%%%%
\subsection{The injection operator $\injectionOp_\level$}
%%%%%%%%%%%%%%%%%%%%%%%%%%%%%%%%%%%%%%%%%%%%%%%%%%%%%%%%%%%%%%%%%%%%%%%%%%%%%%%%%%%%
% 
The difficulty of devising an ``injection operator''
$I_\level: \skeletalSpace_{\level-1} \to \skeletalSpace_\level$ originates from the fact that the finer mesh
has edges which are not refinements of the edges of the coarse
mesh. In \cite{TanPhD}, several possible injection operators are discussed, but turn
out to be unstable. In order to assign reasonable values to these edges, we
construct the injection operator in three steps. First, introduce the continuous finite element space
\begin{gather}
  \contElementSpace_\level := \bigl\{ u \in H^1_0(\Omega)
  \;\big|\;
  u_{|\elem} \in \polynomials_p(\elem) \quad \forall \elem \in \mesh_\level\bigr\}.
  \label{EQ:cg_space}
\end{gather}
We assume that the shape function basis on each mesh cell $T$ is
defined through a Lagrange interpolation condition with respect to
support points $\vec x$ located on vertices, edges, and in the interior of the
cell. Thus, a function in $\contElementSpace_\level$ is uniquely
determined by the values in these support points. We now define the continuous extension operator
\begin{equation}
  \extensionOp_\level\colon \skeletalSpace_{\level} \to \contElementSpace_{\level},
\end{equation}
by the interpolation conditions
\begin{gather}
  \label{EQ:define_uc}
  [\extensionOp_\level\lambda](\vec x) =
  \begin{cases}
    \overline\lambda(\vec x) & \text{if $\vec x$ is on the boundary of a face}, \\
    \lambda(\vec x) & \text{if $\vec x$ is in the interior of a face}, \\
    [\localU_\level\lambda](\vec x) & \text{if $\vec x$ is in the interior of a cell}.
  \end{cases}
\end{gather}
Here, $\overline\lambda$ is the arithmetic mean of the values attained
by $\lambda$ on different faces meeting in $\vec x$.

The spaces $\contElementSpace_{\level}$ are nested, such that there is a
 natural injection operator
\begin{equation}
  \begin{aligned}
    I_\level^c\colon \contElementSpace_{\level-1} &\to \contElementSpace_{\level} \\ u &\mapsto u.
  \end{aligned}
\end{equation}
On $\contElementSpace_{\level}$ the trace on edges is well defined, such that we can write
\begin{equation}
  \begin{aligned}
    \traceOp\colon \contElementSpace_{\level} &\to \skeletalSpace_{\level} \\ u &\mapsto \traceOp u.
  \end{aligned}
\end{equation}
Using these, we define the injection operator $I_\level$ as the
concatenation of extension, natural injection, and trace, namely
\begin{equation}
  \begin{aligned}
    I_\level\colon \skeletalSpace_{\level-1} &\to \skeletalSpace_\level \\
    \lambda &\mapsto \traceOp I_\level^c \extensionOp_{\level-1} \lambda.
  \end{aligned}
\end{equation}
By its definition, $I_\level$ is the operator such that this diagram commutes:
\begin{center}
  \includegraphics[width=.4\textwidth]{tikz/injection.tikz}
\end{center}

Moreover, if $\tau_\level \sim h_\level^{-1}$, the following Lemma can be proved similarly to \cite[Thms.~3.6 and~3.8]{ChenLX2014}:

\begin{lemma}[Boundedness]
  \label{LEM:injection_bounded}
  The injection operator $\injectionOp_\level$ is bounded in the sense that
  \begin{equation}
   a_\level(\injectionOp_\level \lambda , \injectionOp_\level \lambda) \lesssim a_{\level-1}(\lambda, \lambda) \qquad \forall\lambda \in \skeletalSpace_{\level-1}.
  \end{equation}
\end{lemma}

%%%%%%%%%%%%%%%%%%%%%%%%%%%%%%%%%%%%%%%%%%%%%%%%%%%%%%%%%%%%%%%%%%%%%% 
%%%%%%%%%%%%%%%%%%%%%%%%%%%%%%%%%%%%%%%%%%%%%%%%%%%%%%%%%%%%%%%%%%%%%% 
\subsection{Operators for the multigrid method and analysis}
%%%%%%%%%%%%%%%%%%%%%%%%%%%%%%%%%%%%%%%%%%%%%%%%%%%%%%%%%%%%%%%%%%%%%%
%%%%%%%%%%%%%%%%%%%%%%%%%%%%%%%%%%%%%%%%%%%%%%%%%%%%%%%%%%%%%%%%%%%%%% 

  After the discrete operator $A_\level$ and the injection operator
  $\injectionOp_\level$ have been defined, we introduce the remaining
  operators here. First, there are two operators from
  $\skeletalSpace_{\level}$ to $\skeletalSpace_{\level-1}$, which
  replace the $L^2$ projection and the Ritz projection of conforming
  methods, respectively. They are $\projectionOrthogonalOP_{\level-1}$
  and $\projectionOp_{\level-1}$ defined by the conditions
\begin{xalignat}3
   \projectionOrthogonalOP_{\level-1}&\colon \skeletalSpace_\level \to \skeletalSpace_{\level-1},
   &\langle \projectionOrthogonalOP_{\level-1} \lambda, \mu \rangle_{\level-1}
   &= \langle \lambda, \injectionOp_\level \mu\rangle_\level
   && \forall \mu \in \skeletalSpace_{\level-1}.
   \label{EQ:L2_projection_definition}
   \\
   \projectionOp_{\level-1}&\colon \skeletalSpace_\level \to \skeletalSpace_{\level-1},
   &a_{\level-1}(\projectionOp_{\level-1} \lambda, \mu)
   &= a_\level(\lambda, \injectionOp_\level \mu)
   && \forall \mu \in \skeletalSpace_{\level-1},
   \label{EQ:projection_definition}
 \end{xalignat}
 The operator $\projectionOrthogonalOP_{\level-1}$ is used in the
 implementation, while $\projectionOp_{\level-1}$ is key to the
 analysis.

The multigrid operator for preconditioning $A_\level$ will be defined in
  Section \ref{SEC:multigrid_algortith}. It will be referred to as
  \begin{gather}
    B_\level\colon \skeletalSpace_\level \to \skeletalSpace_\level.
 \end{gather}
 It relies on a smoother
 \begin{gather}
   R_\level: \skeletalSpace_\level \to \skeletalSpace_\level,
 \end{gather}
  which can be defined in terms of Jacobi or Gauss-Seidel
  iterations, respectively. Denote by $R_\level^\dagger$ the adjoint operator of
  $R_\level$ with respect to
  $\langle \cdot, \cdot \rangle_\level$ and define $R_\level^\iterMgOuter$ by
 \begin{equation}
  R_\level^\iterMgOuter = \begin{cases} R_\level & \text{ if } \iterMgOuter \text{ is odd,} \\ R_\level^\dagger & \text{ if } \iterMgOuter \text{ is even.} \end{cases}
\end{equation}

At this point, we have defined HDG versions of all operators involved
in standard multigrid analysis. Additionally, we define the averaging operator
\begin{gather}
  \averagingOp_\level\colon\discElementSpace_\level \to \contElementSpace_\level.
\end{gather}
Analog to~\eqref{EQ:define_uc}, it is defined by interpolation in the
support points $\vec x$ of the shape functions of the space $\contElementSpace_\level$, namely
\begin{equation}
 \left[\averagingOp_\level u\right] (\vec x) = \overline u(\vec x).
\end{equation}
Here $\overline u$ is the arithmetic mean of the values $u(\vec x)$ from all mesh cells meeting at $\vec x$. For $\vec x\in \partial\Omega$, we let $\left[\averagingOp_\level u\right] (\vec x) = 0$.

A summary of the different operators connecting the spaces can
be found in Figure \ref{FIG:spaces_and_operators}.
\begin{figure}[b]
 \includegraphics[width=\textwidth]{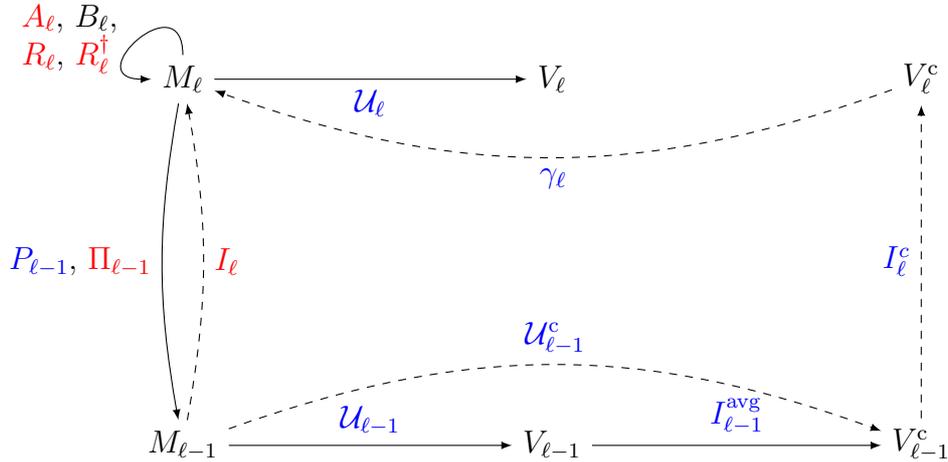}
 \caption{Sketch of the different operators connecting function spaces of refinement levels $\level$ and $\level-1$. Here, the operators needed to implement the multigrid preconditioner $B_\level$ are depicted \textcolor{red}{red}, while operators only appearing in the analysis are in \textcolor{blue}{blue} and spaces are black. The dashed arrows commute, while in general $\extensionOp_{\level-1} \neq \averagingOp_{\level-1} \circ \localU_{\level-1}$.}\label{FIG:spaces_and_operators}
\end{figure}

%

% Note that the spaces $\linElementSpace_\level$ and $\linElementSpace_{\level-1}$ and operators from or to these spaces are introduced for our analysis only and not utilized when executing the method.
%
% ---------------------------------------------------------------------
\section{Multigrid method and main convergence result}\label{SEC:multigrid}
% ---------------------------------------------------------------------
%
We consider a standard (symmetric) V-cycle multigrid method for \eqref{EQ:hdg_condensed} (cf. \cite{BramblePX1991,DuanGTZ2007}) for which we conduct convergence analysis (as done in \cite{DuanGTZ2007}). Thus, Section \ref{SEC:multigrid_algortith} is devoted to illustrating the multigrid method (cf. \cite{DuanGTZ2007} where the algorithm is also taken from), while Section \ref{SEC:main_convergence_result} states the main convergence result.

Let us begin citing an estimate for eigenvalues and condition numbers of the matrices $A_\level$.
\begin{lemma}
 Suppose that $\mesh_\level$ is quasiuniform. Then, there are positive constants $C_1$ and $C_2$ independent of $\level$ such that
 \begin{equation}
  C_1 \| \lambda \|^2_\level \le a_\level ( \lambda, \lambda ) \le \beta_\level C_2 h_\level^{-2} \| \lambda \|^2_\level, \qquad \forall \lambda \in \skeletalSpace_\level,
 \end{equation}
 where $\beta_\level := 1 + (\tau_\level h_\level)^2$.
\end{lemma}
\begin{proof}
 This is a Corollary of Theorem~3.2 in \cite{CockburnDGT2013}.
\end{proof}
This implies that for the stiffness matrix, we can bound the condition number $\kappa_\level$ by
\begin{equation}
 \kappa_\level \lesssim \beta_\level h_\level^{-2}
\end{equation}
which implies that for all choices of $\tau_\level$ satisfying $\tau_\level \lesssim h_\level^{-1}$ the condition number grows like $h_\level^{-2}$. Here and in the following, $\lesssim$ has the meaning of smaller than or equal to up to a constant only dependent on the regularity constant of the mesh family and $c_\text{ref}$.
% 
% ---------------------------------------------------------------------
\subsection{Multigrid algorithm}\label{SEC:multigrid_algortith}
% ---------------------------------------------------------------------
% 
Let $\iterMgInner \in \mathbb N \setminus \{ 0 \}$ be the number of fine-level smoothing steps. We recursively define the multigrid operator of the refinement level $\level$
\begin{equation}
 B_\level \; : \quad \skeletalSpace_\level \to \skeletalSpace_\level,
\end{equation}
by the following steps. Let $B_0 = A^{-1}_0$. For $\level > 0$, let $x^0 = 0 \in \skeletalSpace_\level$. Then for $\mu\in\skeletalSpace_\level$,
\begin{enumerate}
 \item Define $x^\iterMgOuter \in \skeletalSpace_\level$ for $\iterMgOuter = 1, \ldots, \iterMgInner$ by
 \begin{equation}
  x^\iterMgOuter = x^{\iterMgOuter-1} + R_\level^{\iterMgOuter} ( \mu - A_\level x^{\iterMgOuter-1} ).
 \end{equation}
 \item Set $y^0 = x^\iterMgInner + \injectionOp_\level q$, where $q \in \skeletalSpace_{\level-1}$ is defined as
 \begin{equation}
  q = B_{\level-1} \projectionOrthogonalOP_{\level-1} ( \mu - A_\level x^\iterMgInner).
 \end{equation}
 \item Define $y^\iterMgOuter \in \skeletalSpace_\level$ for $\iterMgOuter = 1, \ldots, \iterMgInner$ as
 \begin{equation}
  y^\iterMgOuter = y^{\iterMgOuter - 1} + R^{\iterMgOuter+\iterMgInner}_\level ( \mu - A_\level y^{\iterMgOuter-1} ).
\end{equation}
\item  Let $B_\level \mu = y^{\iterMgInner}$.
\end{enumerate}
% 
% ---------------------------------------------------------------------
\subsection{Main convergence result}\label{SEC:main_convergence_result}
% ---------------------------------------------------------------------
%
The analysis of the multigrid method is based on the framework
introduced in~\cite{DuanGTZ2007}. There, convergence is traced back to
three assumptions. Let $\underline \lambda^A_\level$ be the largest eigenvalue of $A_\level$, and
\begin{gather}
 K_\level := \bigl(1 - (1 - R_\level A_\level) (1 - R^\dagger_\level A_\level)\bigr) A^{-1}_\level.
\end{gather}
Then, there exists constants $C_1, C_2, C_3 > 0$ independent of the mesh level $\level$, such that there holds
% In the remainder of this manuscript we show the preconditions to invoke Theorem \ref{TH:main_theorem} (which is Theo.~3.1 of \cite{DuanGTZ2007}), i.e., we need to verify:
% 
\begin{itemize}
\item Regularity approximation assumption:
  \begin{equation}\label{EQ:precond1}
    | a_\level(\lambda - \injectionOp_\level \projectionOp_{\level-1} \lambda, \lambda) |
    \le C_1 \frac{\| A_\level \lambda \|^2_\level}{\underline \lambda^A_\level} \qquad \forall \lambda \in \skeletalSpace_\level. \tag{A1}
  \end{equation}
\item Stability of the ``Ritz quasi-projection'' $\projectionOp_{\level-1}$ and injection $\injectionOp_\level:$
 \begin{equation}\label{EQ:precond2}
  \| \lambda - \injectionOp_\level \projectionOp_{\level-1} \lambda\|_{a_\level} \le C_2 \| \lambda \|_{a_\level} \qquad \forall \lambda \in \skeletalSpace_\level. \tag{A2}
\end{equation}
\item Smoothing hypothesis:
 \begin{equation}\label{EQ:precond3}
  \frac{\| \lambda \|^2_\level}{\underline \lambda^A_\level} \le C_3 \langle K_\level \lambda, \lambda \rangle_\level. \tag{A3}
 \end{equation}
\end{itemize}
Theorem~3.1 in~\cite{DuanGTZ2007} reads
\begin{theorem}\label{TH:main_theorem}
 Assume that \eqref{EQ:precond1}, \eqref{EQ:precond2}, and \eqref{EQ:precond3} hold. Then for all $\level \ge 0$,
 \begin{equation}
  | a_\level ( \lambda - B_\level A_\level \lambda, \lambda ) | \le \delta a_\level(\lambda, \lambda),
 \end{equation}
 where
 \begin{equation}
  \delta = \frac{C_1 C_3}{\iterMgInner - C_1 C_3}
\end{equation}
with the number of smoothing steps $\iterMgInner > 2 C_1 C_3$.
\end{theorem}
Thus, in order to prove uniform convergence of the multigrid method, we will now set out to verify these assumptions.
% 
% ---------------------------------------------------------------------
\section{Proof of \eqref{EQ:precond1}}\label{SEC:proof_a1}
% ---------------------------------------------------------------------
% 
The main statement of this section is
\begin{theorem}\label{TH:proof_of_A1}
  If \eqref{EQ:diffusion_mixed} has elliptic regularity and
  $\tau_\level \cong h^{-1}_{\level}$, then~\eqref{EQ:precond1} is
  satisfied.
\end{theorem}
% 

% ---------------------------------------------------------------------
\subsection{Preliminaries}
% ---------------------------------------------------------------------
% 
We begin the analysis with some basic results on the injection operator $\injectionOp_\level$ and the ``Ritz quasi-projection'' $\projectionOp_{\level-1}$. Afterwards, we deal with ``quasi-orthogonality'', before we close the section analyzing the ``reconstruction approximation''.

%%%%%%%%%%%%%%%%%%%%%%%%%%%%%%%%%%%%%%%%%%%%%%%%%%%%%%%%%%%%%%%%%%%%%%
\begin{lemma}[Stability]\label{LEM:projection_stable}
  The ``Ritz quasi-projection''
  $\projectionOp_{\level-1}: \skeletalSpace_\level \to
  \skeletalSpace_{\level-1}$ is stable in the sense that for all
  $\lambda \in \skeletalSpace_\level$, we have
  \begin{gather}
     \| \projectionOp_{\level-1}  \lambda \|_{a_{\level-1}} \lesssim \| \lambda \|_{a_\level}.
  \end{gather}
\end{lemma}
\begin{proof}
 From \eqref{EQ:projection_definition}, we can deduce that
 \begin{equation}
  \| \projectionOp_{\level-1} \lambda \|^2_{a_{\level-1}} = a_\level(\lambda, \injectionOp_\level \projectionOp_{\level-1} \lambda ) \le \| \lambda \|_{a_\level} \| \injectionOp_\level \projectionOp_{\level-1} \lambda \|_{a_\level} \lesssim \| \lambda \|_{a_\level} \| \projectionOp_{\level-1} \lambda \|_{a_{\level-1}},
 \end{equation}
 where we used the Cauchy--Schwarz inequality for $a_{\level}(.,.)$ and the boundedness of $I_\level$ from Lemma \ref{LEM:injection_bounded}.
\end{proof}
% 

%%%%%%%%%%%%%%%%%%%%%%%%%%%%%%%%%%%%%%%%%%%%%%%%%%%%%%%%%%%%%%%%%%%%%% 

\begin{lemma}
  \label{LEM:extract_hk_from_U}
 The DG reconstructions of the injection operator admits the estimate
 \begin{equation}
   \| \localU_{\level-1} \mu - \localU_\level \injectionOp_\level \mu \|_0 \lesssim h_\level \| \mu \|_{a_{\level-1}},
   \qquad
   \forall \mu \in \skeletalSpace_{\level-1}.
 \end{equation}
\end{lemma}
\begin{proof}
 We introduce intermediate approximations such that
 \begin{multline}
   \localU_{\level-1} \mu - \localU_\level \injectionOp_\level \mu
   \\= \underbrace{ \localU_{\level-1} \mu - \averagingOp_{\level-1} \localU_{\level-1} \mu }_{ =: \Xi_1 }
   + \underbrace{ \averagingOp_{\level-1} \localU_{\level-1} \mu - \extensionOp_{\level-1} \mu }_{ =: \Xi_2 }
   + \underbrace{ \extensionOp_{\level-1} \mu - \localU_\level \injectionOp_\level \mu }_{ =: \Xi_3 },
 \end{multline}
 and use triangle inequality. For $\Xi_1$, one utilizes the average operator approximation property \cite[(2.28)]{BonitoN2010} and the definition of $\| \cdot \|_{a_\level}$ to obtain
 \begin{equation}
  \| \Xi_1 \|^2_0 \lesssim h_{\level-1} \nnorm \localU_{\level-1} \mu - \mu \nnorm^2_{\level-1} \le h_{\level-1} \frac{\| \mu \|^2_{a_{\level-1}}}{\tau_{\level-1}}.
 \end{equation}
 For the second term, similar to \cite[Lem.~3.1]{ChenLX2014} we have
 \begin{equation}
  \| \Xi_2 \|_0 \lesssim h^{1/2}_{\level-1} \nnorm \localU_{\level-1} \mu - \mu \nnorm_{\level-1}  \le h^{1/2}_{\level-1} \frac{\| \mu \|_{a_{\level-1}}}{\tau^{1/2}_{\level-1}}.
 \end{equation}
 If $p = 1$, $\Xi_3 = 0$, since $\injectionOp_\level \mu$ is the restriction of a continuous function and $\localU_\level$ simply recovers $\extensionOp_{\level-1} \mu$ in this case. If $p \ge 2$, we can use \cite[Theo.~3.8]{ChenLX2014} to conclude
 \begin{equation}
  \| \Xi_3 \|_0 \lesssim h^2_\level \| \Delta \extensionOp_{\level-1} \mu \|_0 \lesssim h_\level \| \nabla \extensionOp_{\level-1} \mu \|_0 \lesssim h_\level \| \mu \|_{a_{\level-1}},
 \end{equation}
 where the last inequality is \cite[Lem.~3.4]{ChenLX2014}.
\end{proof}

%%%%%%%%%%%%%%%%%%%%%%%%%%%%%%%%%%%%%%%%%%%%%%%%%%%%%%%%%%%%%%%%%%%%%%

In the following lemmas, we use the continuous linear finite element space
\begin{equation}
 \linElementSpace_\level~:=~ \{ u \in C(\Omega) \; : \; u|_\elem \in \polynomials_1(\elem) \;\; \forall \elem \in \mesh_\level \text{ and } u = 0 \text{ on } \partial \Omega \}
\end{equation}
to show quasi-orthogonality and a reconstruction approximation property of the method.

\begin{lemma}[Quasi-orthogonality]
  \label{LEM:insert_gradients}
  For all $\lambda \in \skeletalSpace_\level$, we have that
  \begin{equation}
    (\nabla w, \localQ_\level \lambda - \localQ_{\level-1} \projectionOp_{\level-1} \lambda)_0 = 0,
    \qquad \forall w \in \linElementSpace_{\level-1}.
  \end{equation}
\end{lemma}
\begin{proof}
 For $w \in \linElementSpace_{\level-1}$ define $\mu := \gamma_{\level-1} w \in \skeletalSpace_{\level-1}$. By definition, we immediately obtain
 \begin{equation}
   \begin{split}
     \injectionOp_\level \mu &= \traceOp w, \\
     \localQ_{\level-1} \mu &= \localQ_\level \injectionOp_\level \mu = - \nabla w, \\
     \localU_{\level-1} \mu &= \localU_\level \injectionOp_\level \mu = w.
   \end{split}
 \end{equation}
 From the definitions of $a_\level$ and $a_{\level-1}$ we obtain
 \begin{align}
   a_{\level-1}(\projectionOp_{\level-1} \lambda, \mu)
    = & (\localQ_{\level-1} \projectionOp_{\level-1} \lambda, \localQ_{\level-1} \mu)_0
     \notag\\
    & \phantom{=} + \tau_{\level-1} \llangle \localU_{\level-1} \projectionOp_{\level-1} \lambda - \projectionOp_{\level-1} \lambda, \underbrace{ \localU_{\level-1} \mu - \mu }_{ = 0 } \rrangle_{\level-1},\\
   a_\level(\lambda, \injectionOp_\level \mu)
   & = (\localQ_\level \lambda, \localQ_\level \injectionOp_\level \mu)_0 + \tau_\level  \llangle \localU_\level \lambda - \lambda, \underbrace{ \localU_\level \injectionOp_\level \mu - \injectionOp_\level \mu }_{ = 0} \rrangle_\level.
 \end{align}
 By definition of $\projectionOp_{\level-1}$ in~\eqref{EQ:projection_definition}, these two terms are equal and thus we obtain the claimed result as their difference.
\end{proof}

%%%%%%%%%%%%%%%%%%%%%%%%%%%%%%%%%%%%%%%%%%%%%%%%%%%%%%%%%%%%%%%%%%%%%%

\begin{lemma}[Reconstruction approximation]
  \label{LEM:extract_hk}
 If \eqref{EQ:diffusion_mixed} has elliptic regularity and $\tau_\level \cong h_\level^{-1}$, then for all $\lambda \in \skeletalSpace_\level$ there exists an auxiliary function $\ureconstructed \in\linElementSpace_{\level-1}$ such that
 \begin{equation}
   \| \localQ_\level \lambda + \nabla \ureconstructed \|_0
   + \| \localQ_{\level-1} \projectionOp_{\level-1} \lambda + \nabla \ureconstructed \|_0 \lesssim h_\level \| A_\level \lambda \|_\level.
 \end{equation}
\end{lemma}

The proof of this lemma is based on an explicit construction of
$\ureconstructed$ based on the techniques
in~\cite{CockburnDGT2013}. It is conducted in the following
subsection.

%%%%%%%%%%%%%%%%%%%%%%%%%%%%%%%%%%%%%%%%%%%%%%%%%%%%%%%%%%%%%%%%%%%%%%
%%%%%%%%%%%%%%%%%%%%%%%%%%%%%%%%%%%%%%%%%%%%%%%%%%%%%%%%%%%%%%%%%%%%%%
\subsection{Proof of reconstruction approximation}\label{SEC:proof_reco_approx}

Following~\cite{CockburnDGT2013}, we construct several auxiliary
quantities in order to define $\ureconstructed$ in
Lemma \ref{LEM:extract_hk}. First, we define extension operators
$S_{\elem;i}$ on $\skeletalSpace_\level$ for each cell $\elem$ and each
of its faces $\face_i$ into $\polynomials_{p+1}$ on $\elem$ by the
interpolation conditions
\begin{align}
  \langle S_{\elem;i} \lambda, \eta \rangle_{\face_i}
  & = \langle \lambda, \eta \rangle_{\face_i} && \forall \eta \in \polynomials_{p+1}(\face_i),\\
  ( S_{\elem;i} \lambda, v )_\elem
  & = (\localU \lambda, v)_\elem  && \forall v \in \polynomials_p(\elem).
\end{align}
These are used to define the auxiliary inner product
\begin{gather}
  \langle \lambda, \mu \rangle^\heartsuit_\level
  = \sum_{\elem \in \mesh_\level} \frac{1}{\faceNum_\level} \sum_{i = 1}^{\faceNum_\level} \int_\elem S_{\elem;i} \lambda \; S_{\elem;i} \mu \dx,
\end{gather}
and its corresponding norm $\| \cdot \|_{\heartsuit \level}$. Then, for
$\lambda \in \skeletalSpace_\level$ (which is fixed in this section, while $u$ and $\bar u$ depend on it) let
$\phi_\lambda \in \skeletalSpace_\level$ be defined by
\begin{gather}
  \langle \phi_\lambda, \mu \rangle^\heartsuit_\level = a_\level(\lambda, \mu) = \langle A_\level \lambda, \mu \rangle_\level,
  \qquad \mu \in \skeletalSpace_\level.
\end{gather}
Thus, $\phi_\lambda$ represents $A_\level \lambda$ in
$\skeletalSpace_\level$. Similarly,
$f_\lambda = \localU_\level \phi_\lambda\in \discElementSpace_\level$
represents $A_\level \lambda$ on the whole domain. Based on these representations, we define $\tilde u \in H^1_0(\Omega)$ by
\begin{gather}
  (\nabla \tilde u, \nabla v)_0 = (f_\lambda, v)_0,
  \quad \forall v \in H^1_0(\Omega),
\end{gather}
and $\tildelambda_{\level} \in \skeletalSpace_\level$ by
\begin{gather}
  \label{EQ:tilde_lambda}
  a_\level(\tildelambda_{\level}, \mu) = (f_\lambda, \localU_\level \mu)_0 \qquad \forall \mu \in \skeletalSpace_\level.
\end{gather}
In the remainder of this subsection, we show that
$\ureconstructed=\projectionLinOp_{\level-1} \tildelambda_{\level}$ can be used
in Lemma \ref{LEM:extract_hk}.
Here, the Ritz quasi-projection $\projectionLinOp_{\level-1}$ to the cellwise linear coarse space $ \linElementSpace_{\level-1}$ is defined by
\begin{equation}
 \projectionLinOp_{\level-1} \colon \skeletalSpace_\level \to \linElementSpace_{\level-1}, \qquad 
 (\nabla \projectionLinOp_{\level-1} \lambda, \nabla w)_0 = a_\level(\lambda, \traceOp w) \quad \forall w \in \linElementSpace_{\level-1}.
\end{equation}
We begin with an approximation result for $\tildelambda_{\level}$:

%%%%%%%%%%%%%%%%%%%%%%%%%%%%%%%%%%%%%%%%%%%%%%%%%%%%%%%%%%%%%%%%%%%%%%
\begin{lemma}\label{LEM:extract_hk_basis}
 If \eqref{EQ:diffusion_mixed} has elliptic regularity, for all $\lambda \in \skeletalSpace_\level$, we have
 \begin{equation}
   \| \lambda - \tildelambda_{\level} \|_{a_\level}
   \lesssim h_\level \| A_\level \lambda \|_\level \qquad \text{and} \qquad \| f_\lambda \|_0 \lesssim \| A_\level \lambda \|_\level. 
 \end{equation}
\end{lemma}
\begin{proof}
 This is \cite[Lemma~5.10]{TanPhD}.
\end{proof}

%%%%%%%%%%%%%%%%%%%%%%%%%%%%%%%%%%%%%%%%%%%%%%%%%%%%%%%%%%%%%%%%%%%%%%
We now prove the estimate for the second norm in Lemma \ref{LEM:extract_hk}.
To this end, we introduce the auxiliary function $\tildelambda_{\level -1} \in \skeletalSpace_{\level-1}$, which is defined like $\tildelambda_{\level}$ in equation~\eqref{EQ:tilde_lambda}, but on level $\level-1$. Then,
\begin{multline}
  \localQ_{\level-1} \projectionOp_{\level-1} \lambda
  + \nabla \projectionLinOp_{\level-1} \tildelambda_{\level}\\
  =  \underbrace{ \localQ_{\level-1} \projectionOp_{\level-1} \lambda - \localQ_{\level-1} \projectionOp_{\level-1} \tildelambda_{\level} }_{ =: \Xi_1 }
  + \underbrace{ \localQ_{\level-1} \projectionOp_{\level-1} \tildelambda_{\level} - \localQ_{\level-1} \tildelambda_{\level -1} }_{ =: \Xi_2 }
  \label{EQ:proof_split_equation}\\
  + \underbrace{ \localQ_{\level-1} \tildelambda_{\level -1} + \nabla \tilde u }_{ =: \Xi_3}
  + \underbrace{ ( - \nabla \tilde u + \nabla \projectionLinOp_{\level-1} \tildelambda_{\level}) }_{ =: \Xi_4 }.
\end{multline}
 To obtain the result, we now have to bound all four norms $\| \Xi_1 \|_0$ to $\| \Xi_4 \|_0$. First,
 \begin{equation}
  \| \Xi_1 \|_0 \le \| \projectionOp_{\level-1} \lambda - \projectionOp_{\level-1} \tildelambda_{\level} \|_{a_{\level-1}} \lesssim \| \lambda - \tildelambda_{\level} \|_{a_\level} \lesssim h_\level \| A_\level \lambda \|_\level,
 \end{equation}
 where the first inequality follows directly from the definition of $\| \cdot \|_{a_{\level-1}}$, the second inequality is Lemma \ref{LEM:projection_stable}, and the last inequality is Lemma \ref{LEM:extract_hk_basis}. Using the definition of $\tildelambda_{\level -1}$, \eqref{EQ:projection_definition}, and \eqref{EQ:tilde_lambda}, we obtain for $e_{\level-1} := \projectionOp_{\level-1} \tildelambda_{\level} - \tildelambda_{\level -1}$ that
 \begin{equation}
  a_{\level-1} (e_{\level-1} , \mu) = (f_\lambda , \localU_\level \injectionOp_\level \mu - \localU_{\level-1} \mu)_0 \qquad \forall \mu \in \skeletalSpace_{\level-1}.
 \end{equation}
 Choosing $\mu = e_{\level-1}$, reusing that $\| \localQ_{\level-1} \cdot \|_0 \le \| \cdot \|_{a_{\level-1}}$ and exploiting the Cauchy--Schwarz inequality, we get
 \begin{align}\label{EQ:proof_Xi_2}
  \| \Xi_2 \|^2_0 & \le a_\level (e_{\level-1},e_{\level-1}) = (f_\lambda, \localU_\level \injectionOp_\level e_{\level-1} - \localU_{\level-1} e_{\level-1})_0 \\
  & \le \| f_\lambda \|_0 \| \localU_{\level-1} e_{\level-1} - \localU_\level \injectionOp_\level e_{\level-1}  \|_0.
 \end{align}
 The correct bound for $\| \Xi_2 \|_0$ can now be deduced, since Lemma \ref{LEM:extract_hk_from_U} implies that
 \begin{equation}
  \| \localU_{\level-1} e_{\level-1} - \localU_\level \injectionOp_\level e_{\level-1}  \|_0 \lesssim h_\level \| e_{\level-1} \|_{a_{\level-1}}
 \end{equation}
 which allows to extract
 \begin{equation}
  \| e_{\level-1} \|^2_{a_{\level-1}} \lesssim h_\level \| f_\lambda \|_0 \| e_{\level-1} \|_{a_{\level-1}}
 \end{equation}
 from \eqref{EQ:proof_Xi_2}. Dividing this inequality by $\| e_{\level-1} \|_{a_{\level-1}}$ yields %allows to deduce that $\| e_{\level-1} \|_{a_{\level-1}} \lesssim h_\level \| f_\lambda \|_0$ which gives
 \begin{equation}
  \| \Xi_2 \|^2_0 \le \| e_{\level-1} \|^2_{a_{\level-1}} \lesssim h^2_\level \| f_\lambda \|^2_0 \lesssim h^2_\level \| A_\level \lambda \|^2_\level,
 \end{equation}
 where the last inequality follows from Lemma \ref{LEM:extract_hk_basis}. For $\| \Xi_3 \|_0$, we utilize the convergence properties of the HDG method and that $\localQ_{\level-1} \tildelambda_{\level -1} + \localQ_{\level-1} f_\lambda$ approximates $\tilde{\vec q} = - \nabla \tilde u$. Hence, we can deduce that
 \begin{subequations}\label{EQ:proof_for_Xi_3}
 \begin{equation}\label{EQ:xi3_approx_property}
  \| \localQ_{\level-1} \tildelambda_{\level -1} + \localQ_{\level-1} f_\lambda + \nabla \tilde u \|_0 \lesssim h_{\level-1} | \tilde u |_{H^2(\Omega)} \lesssim h_{\level-1} \| f_\lambda \|_0 \lesssim h_{\level} \| A_\level \lambda \|_\level.
 \end{equation}
 Here, the elliptic regularity of \eqref{EQ:diffusion_mixed} enters together with Lemma \ref{LEM:extract_hk_basis} --- which is also needed to deduce
 \begin{equation}
  \| \localQ_{\level-1} f_\lambda \|_0 \lesssim h_{\level-1} \| f_\lambda \|_0 \lesssim h_\level \| A_\level \lambda \|_\level,
 \end{equation}
 \end{subequations}
 where the first inequality is \cite[Lem.~3.7]{ChenLX2014}. \eqref{EQ:proof_for_Xi_3} implies the desired properties for $\| \Xi_3 \|_0$. For the last term, we observe that
 \begin{equation}
  (\nabla \projectionLinOp_{\level-1} \tildelambda_{\level}, \nabla w)_0 = a_\level(\tildelambda_{\level}, \traceOp w) = ( f_\lambda, \localU_\level \traceOp w )_0 = ( f_\lambda, w )_0
 \end{equation}
 for all $w \in \linElementSpace_{\level-1}$. Thus, we can exploit the  approximation property of (continuous) linear finite elements to obtain the result for $\| \Xi_4 \|_0$ similar to \eqref{EQ:xi3_approx_property}.
 
 The inequality for the first term can be obtained analogously substituting \eqref{EQ:proof_split_equation} by
 \begin{equation}
  \localQ_\level  \lambda + \nabla \projectionLinOp_{\level-1} \tildelambda_{\level} = \localQ_\level \lambda - \localQ_\level \tildelambda_{\level} + \localQ_\level \tildelambda_{\level} + \nabla \tilde u - \nabla \tilde u + \nabla \projectionLinOp_{\level-1} \tildelambda_{\level}.
 \end{equation}
Additionally, we use similar techniques as above to prove the following lemma which will be used in the proof of the main convergence result.
\begin{lemma}\label{LEM:u_q_bound}
 If for all $\level$, $\tau_\level h_\level \le c$ for some $c > 0$, we have for all $\lambda \in \skeletalSpace_\level$, $\mu_\level \in \skeletalSpace_\level$, and $\mu_{\level+1} \in \skeletalSpace_{\level+1}$ that
 \begin{gather}
   \begin{split}
     \tau_\level \nnorm \localU_\level \lambda - \lambda \nnorm^2_\level
     &\lesssim \| \localQ_\level \lambda + \nabla \projectionLinOp_{\level-1} \mu_\level \|^2_0 \\
     \tau_\level \nnorm \localU_\level \lambda - \lambda \nnorm^2_\level
     &\lesssim \| \localQ_\level \lambda + \nabla \projectionLinOp_\level \mu_{\level+1} \|^2_0.     
   \end{split}
 \end{gather}
\end{lemma}
\begin{proof}
 Obviously, we have $\localU_\level \gamma_\level \projectionLinOp_{\level-1} \mu_\level = \projectionLinOp_{\level-1} \mu_\level$ and $\localQ_\level \gamma_\level \projectionLinOp_{\level-1} \mu_\level = - \nabla \projectionLinOp_{\level-1} \mu_\level$. Thus
 \begin{multline}
  \tau_\level \nnorm \localU_\level \lambda - \lambda \nnorm^2_\level = \tau_\level \nnorm \localU_\level (\lambda - \gamma_\level \projectionLinOp \mu_\level) - (\lambda - \gamma_\level \projectionLinOp_{\level-1} \mu_\level) \nnorm^2_\level\\
  \lesssim \| \localQ_\level (\lambda - \gamma_\level \projectionLinOp_{\level-1} \mu_\level) \|^2_0 = \| \localQ_\level \lambda + \nabla \projectionLinOp_{\level-1} \mu_\level \|^2_0,
 \end{multline}
 where the inequality is \cite[p.~68]{TanPhD}. The second inequality can be obtained analogously.
\end{proof}
% 
% ---------------------------------------------------------------------
\subsection{Proof of regularity approximation}\label{SEC:proofs_auxiliaries}
% ---------------------------------------------------------------------
% 
With these preliminaries given, we can now state the proof of
Theorem~\ref{TH:proof_of_A1}.  
At first, we recognize that it suffices to show that
 \begin{equation}
  \left| a_\level(\lambda - \injectionOp_\level \projectionOp_{\level-1} \lambda, \lambda) \right| \lesssim h^2_\level \| A_\level \lambda \|^2_\level,
 \end{equation}
 since this is sufficient for \eqref{EQ:precond1} to hold (cf. \cite[Thm.~3.6]{TanPhD}). Using the bilinearity of $a_\level$ and \eqref{EQ:projection_definition}, we immediately obtain
 \begin{align}
  a_\level((\lambda - \injectionOp_\level& \projectionOp_{\level-1} \lambda, \lambda) = a_\level(\lambda, \lambda) - a_{\level-1}(\projectionOp_{\level-1} \lambda, \projectionOp_{\level-1} \lambda) \notag\\
  = & (\localQ_\level \lambda, \localQ_\level \lambda)_0 - (\localQ_{\level-1} \projectionOp_{\level-1} \lambda, \localQ_{\level-1} \projectionOp_{\level-1} \lambda)_0 \tag{T1} \label{EQ:main_proof_T1}\\
  & + \tau_\level \nnorm \localU_\level \lambda - \lambda \nnorm^2_\level - \tau_{\level-1} \nnorm \localU_{\level-1} \projectionOp_{\level-1} \lambda - \projectionOp_{\level-1} \lambda \nnorm^2_{\level-1}, \tag{T2} \label{EQ:main_proof_T2}
 \end{align}
 where the second equation is due to the definition of the bilinear
 forms $a_\level$ and $a_{\level-1}$. Now, we estimate both terms
 \eqref{EQ:main_proof_T1} and \eqref{EQ:main_proof_T2} separately. For
 the first, we use binomial factoring and quasi-orthogonality to
 obtain
 \begin{align}
   \eqref{EQ:main_proof_T1}
  = & (\localQ_\level \lambda + \localQ_{\level-1} \projectionOp_{\level-1} \lambda, \localQ_\level \lambda - \localQ_{\level-1} \projectionOp_{\level-1} \lambda)_0\\
   = & (\localQ_\level \lambda + 2 \nabla \ureconstructed + \localQ_{\level-1} \projectionOp_{\level-1} \lambda, \localQ_\level \lambda - \localQ_{\level-1} \projectionOp_{\level-1} \lambda)_0.
 \end{align}
 Here, $\ureconstructed$ is from Lemma \ref{LEM:extract_hk}.
 Thus,
 \begin{gather}
   \eqref{EQ:main_proof_T1}
   \le \Bigl( \| \localQ_\level \lambda + \nabla \ureconstructed \|_0 + \| \nabla \ureconstructed + \localQ_{\level-1} \projectionOp_{\level-1} \lambda \|_0 \Bigr)
   \; \| \localQ_\level \lambda - \localQ_{\level-1} \projectionOp_{\level-1} \lambda \|_0.
 \end{gather}
 Due to Lemma \ref{LEM:extract_hk}, we can further estimate
 \begin{equation}
   \eqref{EQ:main_proof_T1}\lesssim h_\level \| A_\level \lambda \|_\level
   \Bigl( \| \localQ_\level \lambda + \nabla \ureconstructed \|_0 + \| \nabla \ureconstructed + \localQ_{\level-1} \projectionOp_{\level-1} \lambda \|_0 \Bigr)
 \end{equation}
 Application of Lemma \ref{LEM:extract_hk} gives the desired result for \eqref{EQ:main_proof_T1}.
 
 For \eqref{EQ:main_proof_T2}, we exploit that both summands have exactly the same form and can be treated analogously. Thus, we only demonstrate the procedure for the second summand:
 \begin{equation}
  \tau_{\level-1} \| \localU_{\level-1} \projectionOp_{\level-1} \lambda - \projectionOp_{\level-1} \lambda \|^2_{\level-1} \lesssim \| \localQ_{\level-1} \projectionOp_{\level-1} \lambda + \nabla \ureconstructed \|^2_0 \lesssim h^2_\level \| A_\level \lambda \|^2_\level,
 \end{equation}
 where the first inequality is Lemma \ref{LEM:u_q_bound} and the second inequality is Lemma \ref{LEM:extract_hk}.
% 
% ---------------------------------------------------------------------
\section{Proof of \eqref{EQ:precond2} and \eqref{EQ:precond3}}\label{SEC:proof_a2_a3}
% ---------------------------------------------------------------------
% 
The proof of \eqref{EQ:precond2} is a simple consequence of Lemma \ref{LEM:injection_bounded} with $\projectionOp_{\level-1} \lambda$ instead of $\lambda$ i.e., we have
\begin{align}
 a_\level ( \lambda - & \injectionOp_\level \projectionOp_{\level-1} \lambda, \lambda - \injectionOp_\level \projectionOp_{\level-1} \lambda ) \\
 = & a_\level ( \lambda, \lambda ) - 2 a_\level ( \lambda, \injectionOp_\level \projectionOp_{\level-1} \lambda ) + a_\level ( \injectionOp_\level \projectionOp_{\level-1} \lambda, \injectionOp_\level \projectionOp_{\level-1} \lambda ) \\
 \le & a_\level ( \lambda, \lambda ) \underbrace{ - 2 a_{\level - 1} ( \projectionOp_{\level-1} \lambda,  \projectionOp_{\level-1} \lambda ) }_{ \le 0 } + C \underbrace{ a_{\level-1} ( \projectionOp_{\level-1} \lambda, \projectionOp_{\level-1} \lambda ) }_{ \lesssim \| \lambda \|^2_{a_\level} \; \text{by Lemma \ref{LEM:projection_stable}} },
\end{align}

For the proof of \eqref{EQ:precond3}, we heavily rely on \cite{BrambleP1992} (where \eqref{EQ:precond3} is denoted (2.11)). Theorems 3.1 and 3.2 of \cite{BrambleP1992} ensure that \eqref{EQ:precond3} holds if the subspaces satisfy a ``limited interaction property'' which holds, because each degree of freedom (DoF) only ``communicates'' with other DoFs which are located on the same face as the DoF or on the other faces of the two adjacent elements.

% 
% ---------------------------------------------------------------------
\section{Numerical experiments}\label{SEC:numerics}
% ---------------------------------------------------------------------
% 
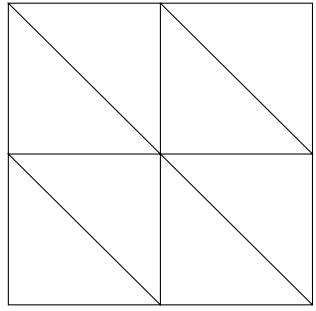
\begin{figure}[t!]
 \begin{tikzpicture}[scale = 2.]
  \draw (0,2) -- (0,0) -- (1,0) -- (1,1) -- (0,1) -- (1,0) -- (2,0) -- (2,2) -- (1,2) -- (1,1) -- (2,1) -- (1,2) -- (0,2) -- (2,0);
 \end{tikzpicture}
 \caption{Coarse grid (level 0) for numerical experiments. Meshes on higher levels are generated by uniform refinement.}\label{FIG:mesh}
\end{figure}
For the numerical evaluation of our multigrid method for HDG, we consider the following Poisson problem on the unit square $\Omega = (0,1)^2$:
\begin{subequations}\label{EQ:num_testcase}
\begin{align}
 -\Delta u & = f && \text{ in } \Omega,\\
 u & = 0 && \text{ on } \partial \Omega,
\end{align}
\end{subequations}
where $f$ is chosen as one. The first mesh is shown in Figure \ref{FIG:mesh} and it is successively refined in our experiments. The implementation is based on the FFW toolbox from \cite{BGGRW07} and employs the Gauss--Seidel smoother. It uses a Lagrange basis and the Euclidean inner product in the coefficient space instead of the inner product $\langle .,. \rangle_\level$. These two inner products are equivalent up to a factor of $h^2_\level$.
Supposing that the matrix form of \eqref{EQ:num_testcase} is $\mathbf A \vec x = \vec b$ we stop the iteration for solving the linear system of equations if
\begin{equation}
 \frac{ \| \vec b - \mathbf A \vec x_\text{iter} \|_2 }{ \| \vec b \|_2 } < 10^{-6}.
\end{equation}
The initial value $\vec x$ on mesh level $\ell$ is the solution on level $\ell-1$, thus we perform a nexted iteration.
The numbers of iteration steps needed are shown in
Table~\ref{TAB:multigrid_steps_f1} for one and two pre- and post-smoothing steps on each level, respectively. Clearly, these numbers
are independent of the mesh level, as predicted by our
analysis. Additionally, the numbers are fairly small, such that we can
conclude that we actually have an efficient method. If we employ two
smoothing steps instead of one, the number of steps is almost divided
by two, such that both options will result in similar numerical
effort. Finally, we see that the choice of
$\tau \in \{ \tfrac{1}{h}, 1%, h
\}$ does not significantly influence the number of iterations. We ran
experiments for polynomial degrees one, two, and three, where
iteration counts remain well bounded; nevertheless, we expect rising
counts for higher degrees, as we use a point smoother.

\begin{table}
 \begin{tabular}{cc|cccccc|cccccc}
  \toprule
  \multicolumn{2}{c|}{smoother}    & \multicolumn{6}{c|}{one step} & \multicolumn{6}{c}{two steps} \\
  \midrule
  \multicolumn{2}{c|}{mesh level}  & 1 & 2 & 3 & 4 & 5 & 6 & 1 & 2 & 3 & 4 & 5 & 6 \\
  \midrule
  \multirow{2}{*}{\rotatebox[origin=c]{90}{$p = 1$}}
  & $\tau = \tfrac{1}{h}$ & 33 & 39 & 38 & 36 & 35 & 35 & 17 & 20 & 19 & 19 & 18 & 18 \\
  & $\tau = 1$            & 33 & 39 & 36 & 35 & 34 & 33 & 17 & 19 & 18 & 18 & 17 & 17\\
  \midrule
  \multirow{2}{*}{\rotatebox[origin=c]{90}{$p = 2$}}
  & $\tau = \tfrac{1}{h}$ & 13 & 12 & 11 & 10 & 10 & 09 & 08 & 07 & 07 & 06 & 06 & 05 \\
  & $\tau = 1$            & 13 & 12 & 11 & 10 & 10 & 09 & 08 & 07 & 07 & 06 & 06 & 05 \\
  \midrule
  \multirow{2}{*}{\rotatebox[origin=c]{90}{$p = 3$}}
  & $\tau = \tfrac{1}{h}$ & 24 & 25 & 25 & 25 & 25 & 25 & 15 & 15 & 15 & 15 & 15 & 15 \\
  & $\tau = 1$            & 24 & 25 & 25 & 25 & 25 & 25 & 15 & 15 & 15 & 15 & 15 & 15 \\
  \bottomrule
 \end{tabular}\vspace{1ex}
 \caption{Numbers of iterations with one and two smoothing steps for $f \equiv 1$. The polynomial degree of the HDG method is $p$.}\label{TAB:multigrid_steps_f1}
\end{table}
% 
% Since the solution in the first example is an eigenfunction of the operator, a second example is conducted with $f \equiv 1$. It is illustrated in Table \ref{TAB:multigrid_steps_f1} and shows that the results are not caused by the eigenfunction property.

Additionally, we tested the correctness of our implementation by
employing a right hand side leading to the solution
$u = \sin(2 \pi x) \sin(2 \pi y)$. The estimated orders of convergence
(EOC) of the primary unknown $u$ computed as
\begin{equation}
 \text{EOC} = \log \left( \frac{ \| u - u_{\level-1} \|_{L^2(\Omega)} }{ \| u - u_\level \|_{L^2(\Omega)} } \right) / \log (2),
\end{equation}
and the secondary unknown $\vec q$ of the HDG method are reported in
Table~\ref{TAB:hdg_eoc}; they coincide well with the orders predicted
in~\cite{CockburnGS2010}.  Iteration counts are almost
identical to those in Table~\ref{TAB:multigrid_steps_f1}, such that we
do not report them here.  We see that the choice $\tau = \tfrac{1}{h}$
is suboptimal as compared to $\tau = 1$, as the error in the secondary
unknown $\vec q$ converges slower by one order. This is why we
included results for $\tau=1$ in Table~\ref{TAB:multigrid_steps_f1}
albeit a theoretical justification is still missing.
\begin{table}
 \begin{tabular}{cc|@{\,}lcc@{\,}lcc@{\,}lcc@{\,}lcc@{\,}lcc@{\,}lcc}
  \toprule
  \multicolumn{2}{c|@{\,}}{mesh}                  && \multicolumn{2}{c}{2}  && \multicolumn{2}{c}{3}   && \multicolumn{2}{c}{4}   && \multicolumn{2}{c}{5}    && \multicolumn{2}{c}{6}     && \multicolumn{2}{c}{7}     \\
  \cmidrule{4-5} \cmidrule{7-8} \cmidrule{10-11} \cmidrule{13-14} \cmidrule{16-17} \cmidrule{19-20}
  \multicolumn{2}{c|@{\,}}{EOC}                   && $u$  & $\vec q$  && $u$  & $\vec q$  && $u$  & $\vec q$  && $u$  & $\vec q$  && $u$  & $\vec q$  && $u$  & $\vec q$  \\
  \midrule
  \multirow{2}{*}{\rotatebox[origin=c]{90}{$p = 1$}}
  & $\tau = \tfrac{1}{h}$ && 1.4 & 1.2 && 1.9 & 1.8 && 2.0 & 1.7 && 2.0 & 1.4 && 2.0 & 1.2 && 2.0 & 1.0 \\
  & $\tau = 1$            && 1.5 & 1.3 && 2.0 & 1.9 && 2.0 & 2.0 && 2.0 & 2.0 && 2.0 & 2.0 && 2.0 & 2.0 \\
%   $\tau = h$            && 1.6 & 1.4 && 2.0 & 1.9 && 2.0 & 2.0 && 2.0 & 2.0 && 2.0 & 2.0 && 2.0 & 2.0 \\
  \midrule
  \multirow{2}{*}{\rotatebox[origin=c]{90}{$p = 2$}}
  & $\tau = \tfrac{1}{h}$ && 3.4 & 3.0 && 3.1 & 2.8 && 3.0 & 2.6 && 3.0 & 2.4 && 3.0 & 2.1 && 3.0 & 2.0 \\
  & $\tau = 1$            && 3.4 & 3.1 && 3.1 & 2.9 && 3.0 & 3.0 && 3.0 & 3.0 && 3.0 & 3.0 && 3.0 & 3.0 \\
%   $\tau = h$            && 3.2 & 3.1 && 3.0 & 2.9 && 3.0 & 3.0 && 3.0 & 3.0 && 3.0 & 3.0 && 3.0 & 3.0 \\
  \midrule
  \multirow{2}{*}{\rotatebox[origin=c]{90}{$p = 3$}}
  & $\tau = \tfrac{1}{h}$ && 3.9 & 2.8 && 4.4 & 3.7 && 4.2 & 3.5 && 4.1 & 3.2 && 4.0 & 3.1 && 4.0 & 3.0 \\
  & $\tau = 1$            && 2.8 & 2.8 && 3.9 & 3.9 && 4.0 & 4.0 && 4.0 & 4.0 && 4.0 & 4.0 && 4.0 & 4.0 \\
  \bottomrule
 \end{tabular}\vspace{1ex}
 \caption{Estimated orders of convergence (EOC) for primary unknown $u$ and secondary unknown $\vec q$ when the polynomial degree of the HDG method is $p$ and the exact solution is $u = \sin(2 \pi x) \sin(2 \pi y)$.}\label{TAB:hdg_eoc}
\end{table}
% 
% ---------------------------------------------------------------------
\section{Conclusions}
% ---------------------------------------------------------------------
% 
We proposed a homogeneous multigrid method for HDG. We proved analytically that this method converges independently of the mesh size. Numerical examples have shown that the condition numbers are not only independent of the mesh size but also reasonably small. As as consequence, we have been enabled to efficiently solve linear systems of equations arising from HDG discretizations of arbitrary order. Our proofs apply to stabilization terms $\tau \sim h^{-1}$, but numerical experiments suggest optimal convergenca also for $\tau \sim 1$.
\bibliographystyle{amsplain}
\bibliography{MultigridHDG}

\providecommand{\bysame}{\leavevmode\hbox to3em{\hrulefill}\thinspace}
\providecommand{\MR}{\relax\ifhmode\unskip\space\fi MR }
\providecommand{\MRhref}[2]{%
  \href{http://www.ams.org/mathscinet-getitem?mr=#1}{#2}
}
\providecommand{\href}[2]{#2}
\begin{thebibliography}{10}

\bibitem{BarrenecheaBosyDoleanNatafTournier18}
G.~Barrenechea, M.~Bosy, V.~Dolean, F.~Nataf, and P.-H. Tournier, \emph{Hybrid
  discontinuous {G}alerkin discretisation and domain decomposition
  preconditioners for the {S}tokes problem}, Comput. Methods Appl. Math.
  \textbf{19} (2018), no.~4, 703--722.

\bibitem{BonitoN2010}
A.~Bonito and R.~H. Nochetto, \emph{Quasi-optimal convergence rate of an
  adaptive discontinuous {G}alerkin method}, SIAM J. Numer. Anal. \textbf{48}
  (2010), no.~2, 734--771.

\bibitem{BrambleP1992}
J.H. Bramble and J.E. Pasciak, \emph{The analysis of smoothers for multigrid
  algorithms}, Math. Comput. \textbf{58} (1992), no.~198, 467--488.

\bibitem{BramblePX1991}
J.H. Bramble, J.E. Pasciak, and J.~Xu, \emph{The analysis of multigrid
  algorithms with nonnested spaces or noninherited quadratic forms}, Math.
  Comput. \textbf{56} (1991), no.~193, 1--34.

\bibitem{Brenner1999}
S.~Brenner, \emph{Convergence of nonconforming multigrid methods without full
  elliptic regularity}, Math. Comput. \textbf{68} (1999), no.~225, 25--53.

\bibitem{BGGRW07}
A.~Byfut, J.~Gedicke, D.~Günther, J.~Reininghaus, and S.~Wiedemann,
  \emph{{FFW} documentation}, https://github.com/project-openffw/openffw.

\bibitem{CelikerCockburnShi10}
F.~Celiker, B.~Cockburn, and K.~Shi, \emph{Hybridizable discontinuous
  {G}alerkin methods for {T}imoshenko beams}, J. Sci. Comput. \textbf{44}
  (2010), no.~1, 1--37.

\bibitem{ChenLX2014}
H.~Chen, P.~Lu, and X.~Xu, \emph{A robust multilevel method for hybridizable
  discontinuous {G}alerkin method for the {H}elmholtz equation}, J. Comput.
  Phys. \textbf{264} (2014), 133--151.

\bibitem{CockburnDongGuzman09}
B.~Cockburn, B.~Dong, and J.~Guzmán, \emph{A hybridizable and superconvergent
  discontinuous {G}alerkin method for biharmonic problems}, J. Sci. Comput.
  \textbf{40} (2009), no.~1-3, 141--187.

\bibitem{CockburnDGT2013}
B.~Cockburn, O.~Dubois, J.~Gopalakrishnan, and S.~Tan, \emph{Multigrid for an
  {HDG} method}, IMA J. Numer. Anal. \textbf{34} (2013), no.~4, 1386--1425.

\bibitem{CockburnG2009}
B.~Cockburn and J.~Gopalakrishnan, \emph{The derivation of hybridizable
  discontinuous {G}alerkin methods for {S}tokes flow}, SIAM J. Numer. Anal.
  \textbf{47} (2009), no.~2, 1092--1125.

\bibitem{CockburnGL2009}
B.~Cockburn, J.~Gopalakrishnan, and R.~Lazarov, \emph{Unified hybridization of
  discontinuous {G}alerkin, mixed, and continuous {G}alerkin methods for second
  order elliptic problems}, SIAM J. Numer. Anal. \textbf{47} (2009), no.~2,
  1319--1365.

\bibitem{CockburnGS2010}
B.~Cockburn, J.~Gopalakrishnan, and F.J. Sayas, \emph{A projection-based error
  analysis of {HDG} methods}, Math. Comput. \textbf{79} (2010), no.~271,
  1351--1367.

\bibitem{CockburnNP2020}
B.~Cockburn, N.~C. Nguyen, and J.~Peraire, \emph{A comparison of {HDG} methods
  for {S}tokes flow}, J. Sci. Comput. \textbf{45} (2020), no.~1, 215--237.

\bibitem{CockburnSayas14}
B.~Cockburn and F.-J. Sayas, \emph{Divergence-conforming {HDG} methods for
  {S}tokes flows}, Math. Comput. \textbf{83} (2014), no.~288, 1571--1598.

\bibitem{DuanGTZ2007}
H.Y. Duan, S.Q. Gao, R.C.E. Tan, and S.~Zhang, \emph{A generalized {BPX}
  multigrid framework covering nonnested {V}-cycle methods}, Math. Comput.
  \textbf{76} (2007), no.~257, 137--152.

\bibitem{EggerWaluga13}
H.~Egger and Ch. Waluga, \emph{{$hp$} analysis of a hybrid {DG} method for
  {S}tokes flow}, IMA J. Numer. Anal. \textbf{33} (2013), no.~2, 687--721.

\bibitem{FengKarakashian01}
X.~Feng and O.~Karakashian, \emph{Two-level non-overlapping {S}chwarz methods
  for a discontinuous {G}alerkin method}, SIAM J. Numer. Anal. \textbf{39}
  (2001), no.~4, 1343--1365.

\bibitem{FuLehrenfeldLinkeStreckenbach20arxiv}
G.~Fu, Ch. Lehrenfeld, A.~Linke, and T.~Streckenbach, \emph{Locking free and
  gradient robust {H(div)}-conforming {HDG} methods for linear elasticity},
  arXiv preprint arXiv:2001.08610 (2020).

\bibitem{Gopa2003}
J.~Gopalakrishnan, \emph{A {S}chwarz preconditioner for a hybridized mixed
  method}, Comput. Methods Appl. Math. \textbf{3} (2003), no.~1, 116--134.

\bibitem{GopaK2003}
J.~Gopalakrishnan and G.~Kanschat, \emph{A multilevel discontinuous {G}alerkin
  method}, Numerische Mathematik \textbf{95} (2003), 527--550.

\bibitem{GopalakrishnanTan09}
J.~Gopalakrishnan and S.~Tan, \emph{A convergent multigrid cycle for the
  hybridized mixed method}, Numer. Lin. Alg. Appl. \textbf{16} (2009),
  689--714.

\bibitem{HeRS2020}
Y.~He, S.~Rhebergen, and H.~De~Sterck, \emph{Local {F}ourier analysis of
  multigrid for hybridized and embedded discontinuous {G}alerkin methods},
  arXiv preprint arXiv:2006.11433 (2020).

\bibitem{HuangHuang19}
Jianguo Huang and Xuehai Huang, \emph{A hybridizable discontinuous {G}alerkin
  method for {K}irchhoff plates}, J. Sci. Comput. \textbf{78} (2019), no.~1,
  290--320.

\bibitem{LehrenfeldSchoeberl16}
Ch. Lehrenfeld and J.~Schöberl, \emph{High order exactly divergence-free
  hybrid discontinuous {G}alerkin methods for unsteady incompressible flows},
  Computer Methods Appl. Mech. Engrg. \textbf{307} (2016), 339 -- 361.

\bibitem{li2016bpx}
B.~Li and X.~Xie, \emph{{BPX} preconditioner for nonstandard finite element
  methods for diffusion problems}, SIAM Journal on Numerical Analysis
  \textbf{54} (2016), no.~2, 1147--1168.

\bibitem{li2016analysis}
B.~Li, X.~Xie, and S.~Zhang, \emph{Analysis of a two-level algorithm for {HDG}
  methods for diffusion problems}, Comm. Comput. Phys. \textbf{19} (2016),
  no.~5, 1435--1460.

\bibitem{NguyenPC2010}
N.C. Nguyen, J.~Peraire, and B.~Cockburn, \emph{A hybridizable discontinuous
  {G}alerkin method for {S}tokes flow}, Computer Methods Appl. Mech. Engrg.
  \textbf{199} (2010), no.~9, 582--597.

\bibitem{oikawa2016analysis}
I.~Oikawa, \emph{Analysis of a reduced-order {HDG} method for the {S}tokes
  equations}, J. Sci. Comput. \textbf{67} (2016), no.~2, 475--492.

\bibitem{SchoeberlLehrenfeld13}
J.~Schöberl and Ch. Lehrenfeld, \emph{Domain decomposition preconditioning for
  high order hybrid discontinuous {G}alerkin methods on tetrahedral meshes},
  Advanced Finite Element Methods and Applications (Berlin, Heidelberg) (Th.
  Apel and O.~Steinbach, eds.), Lecture Notes in Applied and Computational
  Mechanics, vol.~66, Springer, 2013, pp.~27--56.

\bibitem{TanPhD}
S.~Tan, \emph{Iterative solvers for hybridized finite element methods}, Ph.D.
  thesis, University of Florida, 2009.

\end{thebibliography}
\end{document}